\documentclass[reqno, 12pt]{article}

\pdfoutput=1

\usepackage{enumerate}
\usepackage{latexsym}
\usepackage[centertags]{amsmath}
\usepackage{amsfonts}
\usepackage{amssymb}

\usepackage{amsthm}
\usepackage{newlfont}
\usepackage{graphics}
\usepackage{color}
\usepackage{float}
\usepackage{diagbox}

\usepackage{booktabs}
\usepackage{algorithm}
\usepackage{algorithmicx}
\usepackage{algpseudocode}
\usepackage{url}
\usepackage{hyperref}
\usepackage{longtable}
\usepackage{rotating}
\usepackage{multirow}
\usepackage{extarrows}
\usepackage[sort,compress,numbers]{natbib}
\usepackage[utf8]{inputenc}

\newtheorem{thm}{Theorem}[section]
\newtheorem{cor}[thm]{Corollary}
\newtheorem{lem}[thm]{Lemma}
\newtheorem{prop}[thm]{Proposition}

\theoremstyle{mydefinition}
\newtheorem{dfn}[thm]{Definition}
\theoremstyle{myremark}

\newtheorem{exa}[thm]{Example}
\newtheorem{Question}[thm]{Question}

\allowdisplaybreaks[4]

\makeatletter
\let\c@algorithm\c@thm
\makeatother
\numberwithin{algorithm}{section}

\def\Tm{\mathcal{T}_m}

\def\conv{\operatorname{conv}}

\renewcommand{\P}{{\mathbb{P}}}

\def\sgn{\mathrm{sgn}}

\title{The Sign Pattern Problem for Ehrhart Polynomials}

\author{Feihu Liu$^{\color{blue} \dag}$, Sihao Tao$^{\color{blue} \S}$\thanks{Corresponding author}, and Guoce Xin$^{\color{blue} \P}$
\\[2mm]
{\small $^{\color{blue} \dag, \S, \P}$ School of Mathematical Sciences,}\\[-0.8ex]
{\small Capital Normal University, Beijing, 100048, P.R.~China}\\
{\small {\color{blue} $^\dag$} Email address: liufeihu7476@163.com}\\
{\small {\color{blue} $^\S$} Email address: sihao\_tao@cnu.edu.cn}\\
{\small {\color{blue} $^\P$} Email address: guoce\_xin@163.com}
}

\date{\today}

\begin{document}

\maketitle

\begin{abstract}
We investigate the sign patterns of coefficients in the Ehrhart polynomial of the Cartesian product between the $r$-th pyramid over the Reeve tetrahedron and the hypercube $[0, n]^n$. This investigation yields partial results on the sign pattern problem for Ehrhart polynomials. Moreover, we show that for each dimension $d \geq 4$, there exists a $d$-dimensional integral polytope $\mathcal{P}$ such that arbitrarily many of the low-degree coefficients in the Ehrhart polynomial $i(\mathcal{P}, t)$ are negative, while all higher-degree coefficients are positive. Finally, we establish five embedding theorems that enable the sign pattern of a lower-dimensional integral polytope to be embedded into a higher-dimensional integral polytope in various ways. As an application, we completely resolve the Ehrhart coefficient sign pattern problem for dimensions $d = 7, 8, 9$.
\end{abstract}

\noindent
\begin{small}
\end{small}

\noindent
\begin{small}
\end{small}

\section{Introduction}
An \emph{integral} convex polytope is defined as one whose vertices all have integer coordinates.
Let $\mathcal{P}$ be a $d$-dimensional integral convex polytope in $\mathbb{R}^n$.
The \emph{Ehrhart function}
\[
i(\mathcal{P},t) = |t\mathcal{P} \cap \mathbb{Z}^n|, \quad t=1,2,\ldots,
\]
counts the number of integer points in the $t$-th dilation $t\mathcal{P} = \{ t\alpha : \alpha \in \mathcal{P} \}$ of $\mathcal{P}$.

Ehrhart~\cite{Ehrhart62} established that $i(\mathcal{P},t)$ is a polynomial in $t$ of degree $d$ with constant term $1$; this is now known as the \emph{Ehrhart polynomial} of $\mathcal{P}$.
The coefficient of $t^d$ equals the (relative) volume of $\mathcal{P}$, while the coefficient of $t^{d-1}$ is half the boundary volume of $\mathcal{P}$ (see~\cite[Corollary 3.20; Theorem 5.6]{BeckRobins}).
Consequently, the coefficients of $t^d$ and $t^{d-1}$ in $i(\mathcal{P},t)$ are always positive.
The remaining coefficients, however, are more intricate and lack a straightforward geometric interpretation~\cite{McMullen77}.

An integral convex polytope $\mathcal{P}$ is called \emph{Ehrhart positive} (or is said to have \emph{Ehrhart positivity}) if all coefficients of its Ehrhart polynomial \(i(\mathcal{P},t)\) are nonnegative. In dimension $d=3$, a classical example lacking Ehrhart positivity is Reeve's tetrahedron (see~\cite{Reeve} or~\cite[Example 3.22]{BeckRobins}).
\begin{exa}{\em \cite[Example 3.22]{BeckRobins}}\label{ReeveTetrahedron}
Let $\mathcal{T}_m \subseteq \mathbb{R}^3$ be the tetrahedron with vertices $(0,0,0)$, $(1,0,0)$, $(0,1,0)$, and $(1,1,m)$, where $m$ is a positive integer. Its Ehrhart polynomial is given by
\begin{equation}\label{equ_tetrahedron}
i(\mathcal{T}_m,t) = \frac{m}{6}t^3 + t^2 + \frac{12-m}{6}t + 1.
\end{equation}
This tetrahedron is known as the \emph{Reeve tetrahedron}.
Observe that the linear coefficient is negative for $m \ge 13$.
\end{exa}

The study of Ehrhart positivity has attracted considerable attention.
Many families of polytopes have been shown to be Ehrhart positive, including hypersimplices \cite{Hypersimplices}, minimal matroids\cite{MinimalMatroids}, rank two matroids \cite{Rank-Two-Matroids}, cross-polytopes \cite[Section 2]{FuLiu19}, the $y$-families of generalized permutohedra \cite{PostnikovIMRN} which includes the case of Pitman-Stanley polytopes \cite{Stanley-Pitman}, and cyclic polytopes \cite{LiuFuA-M}.
For a comprehensive overview of Ehrhart positivity, we refer to Liu's survey \cite{FuLiu19}.
Many polytopes have also been proven to be non-Ehrhart positive, such as order polytopes \cite{LiuTsuchiya19,LiuXinZhang}, matroid polytopes \cite{MatroidsPolytope}, and smooth polytopes \cite{Smoothpolytopes}.
Next, we assume throughout that $d \geq 3$, since integral polytopes of dimensions $1$ and $2$ are always Ehrhart positive.

Inspired by the Reeve tetrahedron, Hibi, Higashitani, Tsuchiya, and Yoshida \cite{Hibi-Higashitani-Tsuchiya-Yoshida} investigated possible sign patterns of coefficients in Ehrhart polynomials, posing the following question:
\begin{Question}{\em \cite[Question 4.1]{Hibi-Higashitani-Tsuchiya-Yoshida}; \cite[Question 4.2.1]{FuLiu19}; \cite[Conjecture 4.3]{Ferroni23}}\label{QuestionHHTY}
Given $d \geq 3$ and integers $1 \leq i_1 < \cdots < i_q \leq d-2$, does there exist a $d$-dimensional integral polytope $\mathcal{P}$ such that the coefficients of $t^{i_1}, \ldots, t^{i_q}$ in $i(\mathcal{P}, t)$ are negative, while all other coefficients are positive?
\end{Question}
Hibi, Higashitani, Tsuchiya, and Yoshida provided partial answers:
\begin{thm}{\em \cite[Theorem 1.1; Theorem 1.2]{Hibi-Higashitani-Tsuchiya-Yoshida}}\label{Hibi-HTY-Results}
Let $d \geq 3$. Then:
\begin{enumerate}
    \item There exists a $d$-dimensional integral polytope $\mathcal{P}$ such that the coefficients of $t,t^2,\ldots,t^{d-2}$ in $i(\mathcal{P},t)$ are all negative.
    \item For each $1 \leq k \leq d-2$, there exists a $d$-dimensional integral polytope $\mathcal{P}$ such that the coefficient of $t^k$ in $i(\mathcal{P},t)$ is negative, and all other coefficients are positive.
\end{enumerate}
\end{thm}

We call the coefficients of $t,t^2,\ldots,t^{d-2}$ in $i(\mathcal{P},t)$ the \emph{middle Ehrhart coefficients} of $\mathcal{P}$. Hibi et al. established that Question \ref{QuestionHHTY} has an affirmative answer for $3 \leq d \leq 6$ \cite[Proposition 4.2]{Hibi-Higashitani-Tsuchiya-Yoshida}. As noted in Liu's review \cite{FuLiu19}, Tsuchiya further proved that any sign pattern with at most three negative coefficients is realizable among the middle Ehrhart coefficients. Nevertheless, Question \ref{QuestionHHTY} remains open in general.

In this paper, we present three main contributions. First, we investigate sign patterns in the Ehrhart polynomial of the Cartesian product $\mathrm{Pyr}^r(\mathcal{T}_m) \times [0,n]^n$, where $\mathrm{Pyr}^r(\mathcal{T}_m)$ is the $r$-th pyramid over $\mathcal{T}_m$ (Definition \ref{Pyramid-over-P}). This yields partial progress toward Question~\ref{QuestionHHTY}, including the following:
\begin{thm}
Let $\mathcal{P} = \mathrm{Pyr}^r(\mathcal{T}_m) \times [0,n]^n$ be an $(n + r + 3)$-dimensional integral polytope with $n, r \geq 1$. For sufficiently large $m$, the coefficient of $t^j$ in $i(\mathcal{P},t)$ is negative for all $1 \leq j \leq n+1$.
\end{thm}

Our second main result is:
\begin{thm}
Given positive integers $\beta$ and $d$ with $d \geq \beta + 2$, there exists a $d$-dimensional integral polytope $\mathcal{P}$ such that:
\begin{itemize}
    \item The coefficients of $t^{d-2}, t^{d-3}, \ldots, t^{d-\beta-1}$ in $i(\mathcal{P}, t)$ are positive,
    \item The coefficients of $t^{d-\beta-2}, t^{d-\beta-3}, \ldots, t^1$ in $i(\mathcal{P}, t)$ are negative.
\end{itemize}
\end{thm}
The proof is technical; an analogous result is Theorem \ref{lem-more-positive}.

Third, we establish five embedding theorems (Theorems \ref{thm_Embedding Theorem}, \ref{Embedding-Them-II}, \ref{Embedding-III}, \ref{Embedding-IV}, and \ref{Embedding-V}) that allow sign patterns of lower-dimensional integral polytopes to be embedded into higher-dimensional ones via explicit constructions. These results are highly practical. As a first application, we affirm Question \ref{QuestionHHTY} for $d=7$.

For fixed $d \geq 3$, Hibi et al. verified $d-1$ of the $2^{d-2}$ possible sign patterns (Theorem \ref{Hibi-HTY-Results}), leaving $2^{d-2} - (d-1)$ cases unresolved. As a second application, the following theorem shows that verifying only a bounded number of cases suffices.
\begin{thm}
Let $d \geq 5$. Suppose Question~\ref{QuestionHHTY} has been verified for polytopes of dimensions $1$ through $d$. Then for $(d+1)$-dimensional polytopes, it suffices to verify $F_{d+1}$ cases, where $F_n$ satisfies the Fibonacci recurrence $F_n = F_{n-1} + F_{n-2}$ with initial conditions $F_6 = 1$ and $F_7 = 1$.
\end{thm}
Furthermore, we affirm Question \ref{QuestionHHTY} for $d=8$ and $d=9$.

The paper is organized as follows. Section 2 introduces basic definitions and derives the Ehrhart polynomial of $\mathrm{Pyr}^r(\mathcal{T}_m)$. Sections 3--6 present our main contributions.
Note that the results in sections 5 and 6 do not subsume those from the earlier sections.
We compute some specific Ehrhart polynomials using LattE \cite{DeLoera04}.

\section{Preliminary}
This section establishes definitions and results required for subsequent arguments.

\begin{dfn}
Let $\mathcal{P} \subseteq \mathbb{R}^n$ and $\mathcal{Q} \subseteq \mathbb{R}^m$ be convex polytopes. The \emph{Cartesian product} of $\mathcal{P}$ and $\mathcal{Q}$ is the convex polytope defined by
\[
\mathcal{P} \times \mathcal{Q} = \left\{ (x,y) \in \mathbb{R}^{n+m} \mid x \in \mathcal{P},\  y \in \mathcal{Q} \right\}.
\]
\end{dfn}

\begin{prop}{\em \cite[Lemma 2.3]{Hibi-Higashitani-Tsuchiya-Yoshida}}\label{prop_Cartesian}
Let $\mathcal{P} \subseteq \mathbb{R}^n$ and $\mathcal{Q} \subseteq \mathbb{R}^m$ be integral convex polytopes of dimensions $d_1$ and $d_2$, respectively. Then $\mathcal{P} \times \mathcal{Q}$ is an integral polytope of dimension $d_1 + d_2$. Moreover, its Ehrhart polynomial satisfies
\[
i(\mathcal{P} \times \mathcal{Q}, t) = i(\mathcal{P}, t) \cdot i(\mathcal{Q}, t).
\]
\end{prop}

We introduce the following class of polytopes for use in later proofs.
\begin{exa}\label{Linear-GHCube}
Let $\ell_m = [0,m] = \{ \alpha \in \mathbb{R} \mid 0 \leq \alpha \leq m \}$ be the $1$-dimensional integral convex polytope (a closed interval). The $d$-dimensional \emph{generalized hypercube} is defined as $\ell_m^d = [0,m]^d$.
By Proposition \ref{prop_Cartesian}, its Ehrhart polynomial is
\[i(\ell_m^d, t) = (m t + 1)^d.\]
\end{exa}

\begin{dfn}\label{Pyramid-over-P}
Let $\mathcal{P} \subseteq \mathbb{R}^n$ be a polytope with vertices $v_1, \ldots, v_s$. The \emph{pyramid} over $\mathcal{P}$ is defined as
\[
\mathrm{Pyr}(\mathcal{P}) = \mathrm{conv}\left\{ (v_1,0),\ldots,(v_s,0), (0,\ldots,0,1) \right\} \subseteq \mathbb{R}^{n+1},
\]
where $\mathrm{conv}$ denotes the convex hull. For $k \geq 2$, the \emph{$k$-th pyramid} over $\mathcal{P}$, denoted by $\mathrm{Pyr}^k(\mathcal{P})$, is defined recursively as $\mathrm{Pyr}^k(\mathcal{P}) = \mathrm{Pyr}(\mathrm{Pyr}^{k-1}(\mathcal{P}))$.
\end{dfn}

The generating function
\[
\mathrm{Ehr}_{\mathcal{P}}(z) = 1 + \sum_{t \geq 1} i(\mathcal{P}, t) z^t
\]
is called the \emph{Ehrhart series} of $\mathcal{P}$. This can be expressed as
\[
\mathrm{Ehr}_{\mathcal{P}}(z) = \frac{h^*(z)}{(1-z)^{d+1}},
\]
where $\dim \mathcal{P} = d$ and $h^*(z)$ is a polynomial of degree at most $d$ (see \cite[Chapter 3.5]{BeckRobins}), known as the \emph{$h^*$-polynomial} of $\mathcal{P}$.

\begin{prop}{\em \cite[Theorem 2.4]{BeckRobins}}\label{EhrhartSeries-of-Pyr}
For an integral polytope $\mathcal{P} \subseteq \mathbb{R}^n$,
\[
i(\mathrm{Pyr}(\mathcal{P}), t) = 1 + \sum_{j=1}^{t} i(\mathcal{P}, j)
\quad \text{and} \quad
\mathrm{Ehr}_{\mathrm{Pyr}(\mathcal{P})}(z) = \frac{1}{1-z} \mathrm{Ehr}_{\mathcal{P}}(z).
\]
\end{prop}

\begin{prop}\label{Ehrhart-Pyr-Tm}
Let $\Tm$ be the Reeve tetrahedron from Example \ref{ReeveTetrahedron}.
Then the $k$-th pyramid $\mathrm{Pyr}^k(\Tm)$ over $\Tm$ is a $(k+3)$-dimensional integral polytope.
Its Ehrhart polynomial is given by
\[
i(\mathrm{Pyr}^k(\Tm), t) = \binom{t + k + 3}{k + 3} + (m - 1) \binom{t + k + 1}{k + 3}.
\]
\end{prop}

\begin{proof}
Since $\Tm$ is a $3$-dimensional integral polytope, its $k$-th pyramid $\mathrm{Pyr}^k(\Tm)$ is a $(k+3)$-dimensional integral polytope. The Ehrhart series of $\Tm$ is
\[
\mathrm{Ehr}_{\Tm}(z) = 1 + \sum_{t \geq 1} i(\Tm, t) z^t = \frac{h^*(z)}{(1 - z)^4},
\]
where the $h^*$-polynomial of $\Tm$ can be computed by Equation~\eqref{equ_tetrahedron}:
\[
h^*(z) = (1 - z)^4 \left(1 + \sum_{t \geq 1} i(\Tm, t) z^t\right) = 1 + (m - 1)z^2.
\]
Applying Proposition \ref{EhrhartSeries-of-Pyr} iteratively, the Ehrhart series of $\mathrm{Pyr}^k(\Tm)$ is
\[
\mathrm{Ehr}_{\mathrm{Pyr}^k(\Tm)}(z) = \frac{\mathrm{Ehr}_{\Tm}(z)}{(1 - z)^k} = \frac{1 + (m - 1)z^2}{(1 - z)^{4 + k}}.
\]
Using the generating function expansion
\[
\frac{1}{(1 - z)^{d}} = \sum_{t \geq 0} \binom{t + d - 1}{d - 1} z^t,
\]
we obtain
\[
\frac{1 + (m - 1)z^2}{(1 - z)^{4 + k}} = \sum_{t \geq 0} \binom{t + k + 3}{k + 3} z^t + (m - 1) \sum_{t \geq 0} \binom{t + k + 1}{k + 3} z^t.
\]
Therefore, the Ehrhart polynomial of $\mathrm{Pyr}^k(\Tm)$ is
\[
i(\mathrm{Pyr}^k(\Tm), t) = \binom{t + k + 3}{k + 3} + (m - 1) \binom{t + k + 1}{k + 3}.
\]
Note that the binomial coefficient $\binom{t + k + 1}{k + 3}$ vanishes for $t < 2$, consistent with lattice point counts: at $t = 0$, we have $i(\mathrm{Pyr}^k(\Tm), 0) = 1$, and at $t = 1$, the count is $k + 4$, both agreeing with the polynomial expression.
\end{proof}

For $k = 0$, we recover the Ehrhart polynomial of $\Tm$:
\[
i(\mathrm{Pyr}^0(\Tm), t) = \binom{t + 3}{3} + (m - 1) \binom{t + 1}{3} = \frac{m}{6} t^3 + t^2 + \frac{12 - m}{6} t + 1.
\]
For $k = 1$, we obtain the Ehrhart polynomial of the pyramid over $\Tm$:
\[
i(\mathrm{Pyr}^1(\Tm), t) = \binom{t + 4}{4} + (m - 1) \binom{t + 2}{4} = \frac{m}{24} t^4 + \frac{m + 4}{12} t^3 + \frac{36 - m}{24} t^2 + \frac{26 - m}{12} t + 1.
\]

\section{Main Results I}
Before presenting our results, we establish the following lemma.

\begin{lem}\label{lem_B_j}
Let $b$ and $n$ be positive integers with $b \geq n$.
Define the expansion $(bt+1)^n = \sum_{k=0}^{n} B_k t^{k}$ where $B_k = \binom{n}{k} b^{k}$.
Then the sequence $(B_k)_{k=0}^{n-1}$ is strictly increasing, and $B_n \geq B_{n-1}$.
\end{lem}
\begin{proof}
Since $b \geq n$, we have $B_n = b^n \geq n b^{n-1} = B_{n-1}$. For $0 \leq k < n-1$,
\[
\frac{B_{k+1}}{B_k} = \frac{\binom{n}{k+1} b^{k+1}}{\binom{n}{k} b^k} = \frac{n-k}{k+1} \cdot b \geq (n-k)\frac{n}{k+1}>1
\]
implies that $B_{k+1} > B_k$.
\end{proof}

For subsequent arguments, we often treat $B_k = 0$ for $k > n$ or $k < 0$.

\begin{thm}\label{One-Positive}
For $2\leq i\leq 4$, let $d(2)=3, d(3)=4, d(4)=12$.
Given any integer $d \geq d(i)$, there exists a $d$-dimensional integral polytope $\mathcal{P}$ such that the coefficient of $t^{d-2}, \ldots, t^{d-i}$ in $i(\mathcal{P},t)$ is positive, while the coefficient of $t^{j}$ is negative for all $1 \leq j < d-i$.
\end{thm}
\begin{proof}
Since Hibi et al. established the result for $3 \leq d \leq 6$ \cite[Proposition~4.2]{Hibi-Higashitani-Tsuchiya-Yoshida}, we consider $d \geq 7$.

\noindent
\textit{Case 1:} For the case $i=2$, construct the polytope
\[
\mathcal{P} = \mathrm{Pyr}^1(\Tm) \times \ell_n^n \quad \text{with} \quad n \geq 3.
\]
This yields an integral polytope of dimension $n+4 \geq 7$.

By Proposition~\ref{prop_Cartesian},
\[
i(\mathcal{P},t) = \left( \frac{m}{24} t^4 + \frac{m+4}{12} t^3 + \frac{36-m}{24} t^2 + \frac{26-m}{12} t + 1 \right) (nt+1)^n.
\]
Let $i(\mathcal{P},t) = \sum_{j=0}^{n+4} C_j t^j$ and $(nt+1)^n = \sum_{k=0}^{n} B_k t^k$. Then
\begin{align*}
C_j &= B_j + \frac{26-m}{12} B_{j-1} + \frac{36-m}{24} B_{j-2} + \frac{m+4}{12} B_{j-3} + \frac{m}{24} B_{j-4} \\
&= \left( B_j + \frac{13}{6} B_{j-1} + \frac{3}{2} B_{j-2} + \frac{1}{3} B_{j-3} \right)
 - \frac{m}{12} (B_{j-1} - B_{j-3}) - \frac{m}{24} (B_{j-2} - B_{j-4}).
\end{align*}
For $1 \leq j \leq n+2$ with $n \geq 3$ fixed, the term $B_j + \frac{13}{6} B_{j-1} + \frac{3}{2} B_{j-2} + \frac{1}{3} B_{j-3}$ is positive and independent of $m$.
By Lemma~\ref{lem_B_j}, for sufficiently large $m$, we have $C_j < 0$ for all $1 \leq j \leq n+1$.
Finally, we compute
\begin{align*}
C_{n+2} &= \frac{36-m}{24} B_n + \frac{m+4}{12} B_{n-1} + \frac{m}{24} B_{n-2} \\
&= \left( \frac{3}{2} B_n + \frac{1}{3} B_{n-1} \right) + \frac{m}{24} \left( B_n + \binom{n}{2} n^{n-2} \right) > 0,
\end{align*}
since $B_n = n^n$ and $B_{n-1} = n^n$.
This completes the proof for the case $i=2$.

\noindent
\textit{Case 2:} For the case $i=3$, construct
\[
\mathcal{P} = \mathrm{Pyr}^2(\Tm) \times \ell_n^n \quad \text{with} \quad n \geq 2,
\]
yielding an integral polytope of dimension $n+5 \geq 7$.

By Proposition \ref{prop_Cartesian},
\[
i(\mathcal{P},t) = \left( \frac{m}{120} t^5 + \frac{2+m}{24} t^4 + \frac{16+m}{24} t^3 + \frac{46-m}{24} t^2 + \frac{140-3m}{60} t + 1 \right) \cdot (nt+1)^n.
\]
Express $i(\mathcal{P},t) = \sum_{j=0}^{n+5} C_j t^j$ and $(nt+1)^n = \sum_{k=0}^{n} B_k t^k$. The coefficients satisfy
\begin{align*}
C_j &=
B_j + \frac{140-3m}{60} B_{j-1} + \frac{46-m}{24} B_{j-2} + \frac{16+m}{24} B_{j-3} + \frac{2+m}{24} B_{j-4} + \frac{m}{120} B_{j-5} \\
&= \left( B_j + \frac{7}{3} B_{j-1} + \frac{23}{12} B_{j-2} + \frac{2}{3} B_{j-3} + \frac{1}{12} B_{j-4} \right) \\
&\quad - \frac{m}{24} (B_{j-1} - B_{j-3}) - \frac{m}{120} (B_{j-1} - B_{j-5}) - \frac{m}{24} (B_{j-2} - B_{j-4}).
\end{align*}
For $1 \leq j \leq n+3$ with fixed $n \geq 3$, the expression in parentheses is positive and $m$-independent. By Lemma \ref{lem_B_j}, taking $m$ sufficiently large ensures $C_j < 0$ for $1 \leq j \leq n+1$. Finally, we confirm positivity for the highest coefficients:
\begin{align*}
C_{n+3} &= \left( \frac{2}{3} B_n + \frac{1}{12} B_{n-1} \right)
+ \frac{m}{24} B_n + \frac{m}{120} B_{n-1} + \frac{m}{24} B_{n-2} > 0, \\
C_{n+2} &= \left( \frac{23}{12} B_n + \frac{2}{3} B_{n-1} + \frac{1}{12} B_{n-2} \right)
+ \frac{m}{24} B_{n-2} + \frac{m}{120} B_{n-3} > 0.
\end{align*}
This completes the proof for the case $i=3$.

\noindent
\textit{Case 3:} For the case $i=4$, construct the polytope
\[
\mathcal{P} = \mathrm{Pry}^3(\Tm) \times \ell_n^n \quad \text{with} \quad n \geq 6,
\]
which is integral of dimension $d = n+6 \geq 12$. Proposition \ref{prop_Cartesian} gives
\begin{align*}
i(\mathcal{P},t) &= \Bigg[ \frac{m}{720} t^6 + \left( \frac{1}{60} + \frac{m}{80} \right) t^5 + \left( \frac{5}{24} + \frac{5m}{144} \right) t^4 + \left( 1 + \frac{m}{48} \right) t^3 \\
&\quad + \left( \frac{55}{24} - \frac{13m}{360} \right) t^2 + \left( \frac{149}{60} - \frac{m}{30} \right) t + 1 \Bigg] \cdot (nt+1)^n.
\end{align*}

Write $i(\mathcal{P},t) = \sum_{j=0}^{n+6} C_j t^j$ and $(nt+1)^n = \sum_{k=0}^{n} B_k t^k$. Then
\begin{align*}
C_j &=
B_j + \frac{149}{60} B_{j-1} + \frac{55}{24} B_{j-2} + B_{j-3} + \frac{5}{24} B_{j-4} + \frac{1}{60} B_{j-5} \\
&\quad - m \left( \frac{1}{30} B_{j-1} + \frac{13}{360} B_{j-2} - \frac{1}{48} B_{j-3} - \frac{5}{144} B_{j-4} - \frac{1}{80} B_{j-5} - \frac{1}{720} B_{j-6} \right).
\end{align*}
The $m$-independent part is positive for $1 \leq j \leq n+4$ when $n \geq 3$. The coefficient sum in the $m$-dependent term vanishes:
\[
-\frac{1}{30} - \frac{13}{360} + \frac{1}{48} + \frac{5}{144} + \frac{1}{80} + \frac{1}{720} = 0,
\]
reflecting that $t=1$ is a root of the polynomial $\binom{t+k+1}{k+3}$. By Lemma \ref{lem_B_j}, large $m$ ensures $C_j < 0$ for $1 \leq j \leq n+1$. Clearly $C_{n+4} > 0$ and $C_{n+3} > 0$. For $C_{n+2}$:
\begin{align*}
C_{n+2} &= \frac{55}{24} B_n + B_{n-1} + \frac{5}{24} B_{n-2} + \frac{1}{60} B_{n-3} \\
&\quad + m \left( -\frac{13}{360} B_n + \frac{1}{48} B_n + \frac{5}{144} \binom{n}{2} n^{n-2} + \frac{1}{80} \binom{n}{3} n^{n-3} + \frac{1}{720} \binom{n}{4} n^{n-4} \right) \\
&= \frac{55}{24} B_n + B_{n-1} + \frac{5}{24} B_{n-2} + \frac{1}{60} B_{n-3} + \frac{m}{17280} \left( 73n^n - 414n^{n-1} + 83n^{n-2} - 6n^{n-3} \right).
\end{align*}
The expression $73n^n - 414n^{n-1} + 83n^{n-2} - 6n^{n-3}$ is positive for $n \geq 6$, ensuring $C_{n+2} > 0$.
This completes the proof for the case $i=4$.
\end{proof}

However, the pattern identified in Theorem~\ref{One-Positive} fails to extend to the general case.
For instance, consider the integral polytope $\mathcal{P} = \mathrm{Pry}^4(\Tm) \times \ell_n^n$ of dimension $d = n + 7$.
Employing an argument analogous to the proof of Theorem~\ref{One-Positive}, we find that for sufficiently large $m$,
the coefficients of $t^{d-2}$, $t^{d-3}$, and $t^{d-4}$ in $i(\mathcal{P}, t)$ are positive,
while the coefficient of $t^j$ is negative for all $1 \leq j < d - 4$.
Consequently, no positive coefficients beyond those in Theorem~\ref{One-Positive} arise.

Nevertheless, a partial result can be established, see Theorem \ref{Partial-Result-Negative}. We first recall key concepts.
Given a sequence $(a_0, a_1, \ldots, a_n)$ of nonnegative real numbers, it is \emph{unimodal} if there exists $0 \leq i \leq n$ such that
$a_0 \leq \cdots \leq a_i \geq \cdots \geq a_n$. It is \emph{log-concave} if $a_i^2 \geq a_{i-1}a_{i+1}$ for all $1 \leq i \leq n-1$.
A polynomial $a_0 + a_1 t + \cdots + a_n t^n$ is \emph{real-rooted} if all its roots are real.
Comprehensive surveys on these properties are available in \cite{StanleyLog-concave}.

The following classical result, due to Newton, is essential:
\begin{lem}{\em \cite[P.~104]{Hardy-Littlewood}}\label{Realrooted-To-Unimod}
    Let $(a_0, a_1, \ldots, a_n)$ be nonnegative reals. If $a_0 + a_1 t + \cdots + a_n t^n$ is real-rooted,
    the sequence is log-concave and has no internal zeros; in particular, it is unimodal.
\end{lem}

\begin{lem}\label{Posit-Delta-Negat}
    Let $f(t) = (t + r + 1)(t + r) \cdots (t + 1)t(t - 1) = \sum_{j=1}^{r+3} s_j t^j$.
    For any integer $r \geq 1$, there exists an integer $\delta \geq 2$ such that $s_j > 0$ for $j > \delta$ and $s_j < 0$ for $1 \leq j \leq \delta$.
    Additionally, $s_1 < 0$, $s_2 < 0$, and $\sum_{j=1}^{r+3} s_j = 0$.
\end{lem}
\begin{proof}
Since $t = 1$ is a root of $f(t)$, we have $\sum_{j=1}^{r+3} s_j = 0$.
Direct computation yields $s_1 = -(r+1)! < 0$ and
\[s_2 = (r+1)! \left( 1 - \sum_{i=1}^{r+1} \frac{1}{i} \right) < 0,\]
as the harmonic series exceeds $1$. Factorize $f(t) = (t - 1) g(t)$ with
\[
        g(t) = t(t + 1) \cdots (t + r + 1) = \sum_{i=1}^{r+2} c(r + 2, i) t^i,
    \]
where $c(r+2, i)$ are the \emph{unsigned Stirling numbers of the first kind} \cite[Chapter~1]{RP.Stanley}.

The coefficients $c(r+2, i)$ are positive, and $g(t)$ is real-rooted.
By Lemma~\ref{Realrooted-To-Unimod}, the sequence $\big(c(r+2, i)\big)_{i=1}^{r+2}$ is unimodal.
Furthermore, \cite{Hammersley,Erdos1953} established that $c(r+2, i) \neq c(r+2, i-1)$ for all $r$ and $i$,
implying strict unimodality:
\begin{align}\label{Strict-Unimodal}
        c(r+2, 1) < c(r+2, 2) < \cdots < c(r+2, \delta) > \cdots > c(r+2, r+2),
\end{align}
where $\delta$ is the mode.
Expanding $f(t)$ gives
\[
        f(t) = c(r+2, r+2) t^{r+3} + \sum_{i=2}^{r+2} \big( c(r+2, i-1) - c(r+2, i) \big) t^i - c(r+2, 1) t.
\]
From \eqref{Strict-Unimodal}, $c(r+2, i-1) - c(r+2, i) < 0$ for $i \leq \delta$ and $> 0$ for $i > \delta$.
The conclusion follows.
\end{proof}

\begin{exa}
For $f(t) = (t + r + 1) \cdots (t + 1) t (t - 1)$, let $\delta$ be defined as in Lemma~\ref{Posit-Delta-Negat}.
Computations yield:
\begin{align*}
        \delta = 2 \quad & \text{for} \quad 1 \leq r \leq 5; \\
        \delta = 3 \quad & \text{for} \quad 6 \leq r \leq 22; \\
        \delta = 4 \quad & \text{for} \quad 23 \leq r \leq 70; \\
        \delta = 5 \quad & \text{for} \quad 71 \leq r \leq 201.
\end{align*}
\end{exa}

\begin{thm}\label{Partial-Result-Negative}
Let $\mathcal{P} = \mathrm{Pry}^r(\Tm) \times \ell_n^n$ be a $(n+r+3)$-dimensional integral polytope, where $n,r \geq 1$. Then, for sufficiently large $m$, the coefficient of $t^j$ in $i(\mathcal{P},t)$ is negative for all $1 \leq j \leq n+1$.
\end{thm}
\begin{proof}
By Propositions \ref{prop_Cartesian} and \ref{Ehrhart-Pyr-Tm}, we have
$$
i(\mathcal{P},t) = \binom{t + r + 3}{r + 3} \cdot (nt+1)^n + (m-1) \binom{t + r + 1}{r + 3} \cdot (nt+1)^n.
$$
Clearly, $\binom{t + r + 3}{r + 3} \cdot (nt+1)^n$ is a polynomial in $t$ with positive coefficients and is independent of $m$.

Let
\begin{align*}
D(t) &=(m- 1) \binom{t + r + 1}{r + 3}\cdot (nt+1)^n= \sum_{j=0}^{n+r+3} D_j t^j,
\\ (nt+1)^n & = \sum_{k=0}^{n} B_k t^k,
\\ \binom{t + r + 1}{r + 3}&=\frac{(t+r+1)(t+r)\cdots (t+1)t(t-1)}{(r+3)!}=\frac{\sum_{j=1}^{r+3}s_jt^j}{(r+3)!}.
\end{align*}

Comparing coefficients yields
\begin{align*}
D_j=\frac{m-1}{(r+3)!} (s_1B_{j-1}+s_2B_{j-2}+\cdots+s_{r+3}B_{j-(r+3)}),
\end{align*}
where we extend the definition of $B_k$ by setting $B_k =0$ for $k>n$ or $k<0$.

By Lemma \ref{lem_B_j}, we have
$$
n^n = B_n > B_{n-1} > B_{n-2} > \cdots > B_1 > B_0 = 1.
$$
Let $\delta \geq 2$ be the integer given by Lemma \ref{Posit-Delta-Negat}. For $1 \leq j \leq \delta$, since $s_j < 0$ for all $1 \leq j \leq \delta$, it follows that
$$
s_1 B_{j-1} + s_2 B_{j-2} + \cdots + s_{r+3} B_{j-(r+3)} < 0.
$$
For $\delta+1 \leq j \leq n+1$, we estimate
\begin{align*}
s_1 B_{j-1} + s_2 B_{j-2} + \cdots + s_{r+3} B_{j-(r+3)} &\leq (s_1 + \cdots + s_\delta) B_{j-\delta} + (s_{\delta+1} + \cdots + s_{r+3}) B_{j-(\delta+1)} \\
&= -\rho B_{j-\delta} + \rho B_{j-(\delta+1)} = \rho \left( B_{j-(\delta+1)} - B_{j-\delta} \right) < 0,
\end{align*}
where $\rho$ is a positive real number. (Note that if $j - \delta = n$, then $j = n + \delta \geq n + 2$, contradicting $j \leq n + 1$.)

Thus, for sufficiently large $m$, $D_j$ is negative and has sufficiently large absolute value for all $1 \leq j \leq n+1$. Consequently, the coefficient of $t^j$ in $i(\mathcal{P},t)$ is negative for all $1 \leq j \leq n+1$. This completes the proof.
\end{proof}

\section{Main Results II}

In this section, we first prove that for any integer $d \geq 4$, there exists a $d$-dimensional integral polytope $\mathcal{P}$ such that the coefficients of the linear and quadratic terms in $i(\mathcal{P},t)$ are negative, while all other coefficients are positive; see Lemma \ref{lem_two_negative}.
Then, using this result, we establish the existence of a $d$-dimensional integral polytope $\mathcal{P}$ such that several low-degree coefficients in $i(\mathcal{P},t)$ are negative, while all remaining higher-degree coefficients are positive; see Theorems \ref{lem-more-positive} and \ref{Thm-More-positive-negi}.

\begin{dfn}
Let $n \in \mathbb{Z}_+$ be a positive integer. Define a family of polynomials in $t$, parameterized by $m \in \mathbb{Z}_+$, as follows:
\begin{align*}
\mathcal{F}^n(m) = \left\{ 1 + \sum_{i=1}^{n}(b_i + a_i m) t^i \Bigg| \begin{array}{c} a_i, b_i \in \mathbb{R}; \\ a_1, a_2 < 0; \\ a_3, \ldots, a_n > 0; \\ b_i \geq 0\ \text{for all}\ i \end{array} \right\}.
\end{align*}
\end{dfn}

\begin{lem}\label{lem_two_negative}
For any integer $d \geq 4$, there exists a $d$-dimensional integral polytope $\mathcal{P}$ such that $i(\mathcal{P}, t) \in \mathcal{F}^d(m)$.
\end{lem}

\begin{proof}
We proceed by induction on $d$. The base case $d=4$ holds since
$$i(\mathrm{Pry}^1(\Tm) ,t) = \frac{m}{24} t^4 + \frac{m+4}{12} t^3 + \frac{36-m}{24} t^2 + \frac{26-m}{12} t + 1.$$

Assume the induction hypothesis for $d=n$, i.e., there exists an $n$-dimensional integral polytope $\mathcal{P}$ such that
$$i(\mathcal{P},t) = 1 + \sum_{i=1}^{n}(b_i + a_i m) t^i \in \mathcal{F}^n(m).$$

Construct the polytope $\mathcal{Q} = r\mathcal{P} \times \ell_1$ with $r \in \mathbb{Z}_+$, where $\ell_1$ denotes the 1-dimensional unit simplex. Then $\mathcal{Q}$ is an $(n+1)$-dimensional integral polytope. Since $i(r\mathcal{P},t) = i(\mathcal{P},rt)$, we have
\begin{align*}
i(\mathcal{Q},t) &= i(r\mathcal{P},t) \cdot i(\ell_1,t) \\
&= \left(1 + \sum_{i=1}^{n} r^i (b_i + a_i m) t^i \right) \cdot (1 + t) \\
&= 1 + \left(1 + r(b_1 + a_1 m)\right)t + r^n (b_n + a_n m) t^{n+1} \\
&\quad + \sum_{i=2}^{n} \left[ (r^i b_i + r^{i-1} b_{i-1}) + (r^i a_i + r^{i-1} a_{i-1}) m \right] t^i.
\end{align*}
For sufficiently large $r \in \mathbb{Z}_+$, we obtain $i(\mathcal{Q},t) \in \mathcal{F}^{n+1}(m)$. This completes the induction step.
\end{proof}

\begin{thm}\label{Two-negative-Thm-Cpositiv}
For any integer $d \geq 4$, there exists a $d$-dimensional integral polytope $\mathcal{P}$ such that the coefficients of $t$ and $t^2$ in $i(\mathcal{P},t)$ are negative, while those of $t^3, t^4, \ldots, t^{d-2}$ are positive.
\end{thm}
\begin{proof}
By Lemma \ref{lem_two_negative}, the result holds for sufficiently large positive integers $m$.
\end{proof}

\begin{thm}\label{lem-more-positive}
Given arbitrary integers $n \geq 1$ and $k \geq 4$, there exists an $(n+k)$-dimensional integral polytope $\mathcal{P}$ such that the coefficients of $t^{n+k-2}, t^{n+k-3}, \ldots, t^{n+3}$ in $i(\mathcal{P}, t)$ are positive, and the coefficients of $t^{n+2}, t^{n+1}, \ldots, t$ are negative.
\end{thm}
\begin{proof}
By Lemma~\ref{lem_two_negative}, there exists a $k$-dimensional integral polytope $\mathcal{Q}$ such that
\[
i(\mathcal{Q},t) = 1 + \sum_{i=1}^{k} (b_i + a_i m) t^i \in \mathcal{F}^k(m).
\]
Construct the polytope $\mathcal{P} = \mathcal{Q} \times \ell_b^n$ with $b \geq n$. Then $\mathcal{P}$ is an $(n+k)$-dimensional integral polytope. Its Ehrhart polynomial is given by
\begin{align*}
i(\mathcal{P}, t) &= i(\mathcal{Q}, t) \cdot (b t + 1)^n \\
&= \left(1 + \sum_{i=1}^k (b_i + a_i m) t^i\right) \cdot (b t + 1)^n \\
&= \left(1 + \sum_{i=1}^k b_i t^i\right) (b t + 1)^n + \left(\sum_{i=1}^k a_i m t^i\right) (b t + 1)^n.
\end{align*}
Let $i(\mathcal{P}, t) = \sum_{j=0}^{n+k} C_j t^j$ and $(b t + 1)^n = \sum_{i=0}^{n} B_i t^i$, where $B_i = 0$ for $i > n$ or $i < 0$. Then $C_0 = 1$. For $1 \leq j \leq n+k$, comparing coefficients yields
\[
C_j = b_j + \sum_{\ell=1}^{k} b_\ell B_{j-\ell} + m \sum_{\ell=1}^{k} a_\ell B_{j-\ell},
\]
where $b_j = 0$ for $j > k$ (and similarly, terms with invalid indices are zero). It is clear that $C_{n+3}, C_{n+4}, \ldots, C_{n+k-2} > 0$. For sufficiently large $m$, we have $C_1 < 0$ and $C_2 < 0$.

Now consider $3 \leq j \leq n+2$. By Lemma~\ref{lem_B_j}, $\frac{B_{i+1}}{B_i} = \frac{n-i}{i+1} b > 1$ for $0 \leq i \leq n-1$ when $b$ is sufficiently large. Since $a_1, a_2 < 0$ and $a_3, \ldots, a_k > 0$, it follows that for all $j = 3, \ldots, n+2$,
\[
\sum_{\ell=1}^{k} a_\ell B_{j-\ell} < 0
\]
when $b$ is a sufficiently large positive integer; the negative terms $a_1 B_{j-1} + a_2 B_{j-2}$ dominate due to the growth of $B_i$ with $b$. For sufficiently large $b$, the term $b_j + \sum_{\ell=1}^{k} b_\ell B_{j-\ell}$ is a large positive constant independent of $m$. Therefore, for sufficiently large positive integers $m$, we have $C_j < 0$ for all $j = 3, \ldots, n+2$.
This completes the proof.
\end{proof}

\begin{thm}\label{Thm-More-positive-negi}
Given positive integers $\beta$ and $d$ with $d \geq \beta + 2$, there exists a $d$-dimensional integral polytope $\mathcal{P}$ such that in $i(\mathcal{P}, t)$,
the coefficients of $t^{d-2}, t^{d-3}, \ldots, t^{d-\beta-1}$ are positive,
while the coefficients of $t^{d-\beta-2}, t^{d-\beta-3}, \ldots, t^1$ are negative.
\end{thm}
\begin{proof}
Since Hibi et al.\ established the affirmative answer to Question~\ref{QuestionHHTY} for $3 \leq d \leq 6$ \cite[Proposition 4.2]{Hibi-Higashitani-Tsuchiya-Yoshida},
we restrict to $d \geq 7$ and $\beta \geq 1$.
Cases where $\beta = 1, 2$ and $d \geq 7$ follow from Theorem~\ref{One-Positive}.

For $\beta \geq 3$, Theorem~\ref{lem-more-positive} covers $d = \beta + i$ ($i \geq 5$).
The remaining cases are
\[
(\beta,\beta+2), \quad
(\beta,\beta+3), \quad
(\beta,\beta+4) \quad (\beta \geq 3).
\]
These are resolved as follows:
\begin{itemize}
    \item For $(\beta, \beta+2)$: The Ehrhart polynomial of $\mathcal{P} =[0,\beta+2]^{\beta+2}$ has all positive coefficients, satisfying the requirement.

    \item For $(\beta, \beta+3)$: Only the coefficient of $t^1$ is negative, as shown in \cite[Theorem 1.2]{Hibi-Higashitani-Tsuchiya-Yoshida}.

    \item For $(\beta, \beta+4)$: Only the coefficients of $t^2$ and $t^1$ are negative, per Theorem~\ref{Two-negative-Thm-Cpositiv}.
\end{itemize}
This completes the proof.
\end{proof}

\section{Embedding Theorem I and Its Applications}

For the subsequent theorem, define the sign function $\mathrm{sgn}(n)$ for an integer $n$ by $\mathrm{sgn}(n) = -1$ if $n < 0$, $\mathrm{sgn}(n) = 0$ if $n = 0$, and $\mathrm{sgn}(n) = +1$ if $n > 0$. The following embedding theorem preserves the sign pattern of a $d$-dimensional integral polytope when elevating it to dimension $d+1$.

\begin{thm}[Embedding Theorem I]\label{thm_Embedding Theorem}
Let $\mathcal{P}$ be a $d$-dimensional integral polytope with Ehrhart polynomial
$i(\mathcal{P},t) = 1 + \sum_{i=1}^{d} c_i t^i$, where $c_i \neq 0$ for all $i$.
Then there exist $(d+1)$-dimensional integral polytopes $\mathcal{Q}_1$ and $\mathcal{Q}_2$ with Ehrhart polynomials
\[
i(\mathcal{Q}_1, t) = 1 + \sum_{i=1}^{d+1} h_i t^i \quad \text{and} \quad i(\mathcal{Q}_2, t) = 1 + \sum_{i=1}^{d+1} g_i t^i,
\]
satisfying:
\begin{enumerate}
    \item $\mathrm{sgn}(h_i) = \mathrm{sgn}(c_i)$ for $1 \leq i \leq d-1$, and in particular $\mathrm{sgn}(h_{d-1}) = +1$;
    \item $\mathrm{sgn}(g_i) = \mathrm{sgn}(c_{i-1})$ for $2 \leq i \leq d-1$, and in particular $\mathrm{sgn}(g_1) = +1$.
\end{enumerate}
\end{thm}

\begin{proof}
Define the polytopes
\[
\mathcal{Q}_1 = r \mathcal{P} \times \ell_1 \quad \text{and} \quad \mathcal{Q}_2 = \mathcal{P} \times \ell_m,
\]
where $r$ and $m$ are integers chosen sufficiently large. Applying Proposition \ref{prop_Cartesian} yields
\begin{align*}
i(\mathcal{Q}_1, t) &= \left( 1 + \sum_{i=1}^{d} r^i c_i t^i \right) (1 + t)
= 1 + (r c_1 + 1)t + \sum_{i=2}^{d} (r^i c_i + r^{i-1} c_{i-1}) t^i + r^d c_d t^{d+1}, \\
i(\mathcal{Q}_2, t) &= \left( 1 + \sum_{i=1}^{d} c_i t^i \right) (1 + mt)
= 1 + (c_1 + m)t + \sum_{i=2}^{d} (c_i + m c_{i-1}) t^i + m c_d t^{d+1}.
\end{align*}
For sufficiently large $r$ and $m$, the dominant terms govern the signs:
\[
\mathrm{sgn}(r^i c_i + r^{i-1} c_{i-1}) = \mathrm{sgn}(c_i) \quad \text{and} \quad \mathrm{sgn}(c_i + m c_{i-1}) = \mathrm{sgn}(c_{i-1}).
\]
This establishes the claimed sign conditions.
\end{proof}

Hibi, Higashitani, Tsuchiya, and Yoshida \cite[Proposition 4.2]{Hibi-Higashitani-Tsuchiya-Yoshida} affirmatively resolved Question \ref{QuestionHHTY} for dimensions $3 \leq d \leq 6$. We now extend this result to $7$-dimensional integral polytopes.

\begin{prop}
Given integers $i_1,\ldots,i_q$ with $1\leq i_1<\cdots < i_q\leq 5$, there exists a $7$-dimensional integral polytope $\mathcal{P}$ such that the coefficients of $t^{i_1},\ldots,t^{i_q}$ in $i(\mathcal{P}, t)$ are negative, and all remaining coefficients are positive.
\end{prop}
\begin{proof}
Let $\mathcal{P}$ be a $7$-dimensional integral polytope with Ehrhart polynomial
$i(\mathcal{P},t)= 1 + \sum_{i=1}^{7} c_i t^i$, where each $c_i \neq 0$.
In \cite[Proposition 4.2]{Hibi-Higashitani-Tsuchiya-Yoshida}, Hibi et al. constructed all $6$-dimensional polytopes relevant to Question \ref{QuestionHHTY}.
Therefore, by Embedding Theorem I \ref{thm_Embedding Theorem}, it suffices to consider the following eight sign configurations for the Ehrhart coefficients of $7$-dimensional polytopes:
\[
\Big\{ \big( \sgn(c_5),\sgn(c_4),\sgn(c_3),\sgn(c_2),\sgn(c_1) \big) \in \{-1,+1\}^5 \mid \sgn(c_5) = \sgn(c_1) = -1 \Big\}.
\]
The case where $\sgn(c_5) = \sgn(c_4)=\sgn(c_3)=\sgn(c_2) = \sgn(c_1) = -1$ is covered by Theorem \ref{Hibi-HTY-Results}.

Recall the $s$-th dilation of an integral polytope $\mathcal{P}$, denoted $s\mathcal{P}=\{ s\alpha : \alpha \in \mathcal{P} \}$
is an integral polytope of dimension $\dim \mathcal{P}$. We construct polytopes for the remaining seven cases as follows:
\begin{align*}
\mathcal{P}_{+--} &= \mathcal{T}_{10} \times \mathcal{T}_{100} \times \ell_{10}, \\
\mathcal{P}_{-+-} &= 3(\mathcal{T}_{10} \times \mathcal{T}_{100}) \times \ell_4, \\
\mathcal{P}_{--+} &= 10\mathcal{T}_{100} \times \mathrm{Pry}^1(\mathcal{T}_{10000}), \\
\mathcal{P}_{++-} &= \mathcal{T}_{15} \times \mathcal{T}_{150} \times \ell_{18}, \\
\mathcal{P}_{-++} &= \mathcal{T}_{20} \times \mathcal{T}_{21} \times \ell_{1}, \\
\mathcal{P}_{+++} &= \mathrm{conv}\left\{ \mathbf{0}^7, \mathbf{e}_1^7, \mathbf{e}_2^7, \mathbf{e}_3^7, \mathbf{e}_4^7, \mathbf{e}_5^7, \mathbf{e}_6^7, (1,1,1,1000,1000,1000,1001) \right\}, \\
\mathcal{P}_{+-+} &= 99\Big( \mathrm{conv}\left\{ \mathbf{0}^5, \mathbf{e}_1^5, \mathbf{e}_2^5, \mathbf{e}_3^5, \mathbf{e}_4^5, (3,4,5,8,371) \right\} \Big) \times \mathrm{conv}\left\{ \mathbf{0}^2, \mathbf{e}_1^2, (1,1000000), (2,1000000) \right\},
\end{align*}
where $\mathcal{T}_{m}$ and $\ell_{k}$ denote the polytopes from Examples \ref{ReeveTetrahedron} and \ref{Linear-GHCube}, respectively; $\mathbf{e}_i^d$ is the $i$-th standard basis vector in $\mathbb{R}^d$; and $\mathbf{0}^d$ is the origin in $\mathbb{R}^d$.

The corresponding Ehrhart polynomials are:
\begin{align*}
i(\mathcal{P}_{+--},t) &= \tfrac{2500}{9}t^7 + \tfrac{1900}{9}t^6 - \tfrac{1445}{9}t^5 + \tfrac{199}{9}t^4 - \tfrac{224}{9}t^3 - \tfrac{1316}{9}t^2 - \tfrac{13}{3}t + 1, \\
i(\mathcal{P}_{-+-},t) &= 81000t^7 + 38070t^6 - 1341t^5 - 1017t^4 + 4t^3 - 198t^2 - 39t + 1, \\
i(\mathcal{P}_{--+},t) &= \tfrac{62500000}{9}t^7 + \tfrac{125425000}{9}t^6 - \tfrac{62074700}{9}t^5 - \tfrac{126145340}{9}t^4 \\
&\quad - \tfrac{42527}{9}t^3 + \tfrac{2188607}{18}t^2 - \tfrac{5867}{6}t + 1, \\
i(\mathcal{P}_{++-},t) &= 1125t^7 + \tfrac{1115}{2}t^6 - \tfrac{2429}{2}t^5 + 3t^4 + 247t^3 - \tfrac{819}{2}t^2 - \tfrac{11}{2}t + 1, \\
i(\mathcal{P}_{-++},t) &= \tfrac{35}{3}t^7 + \tfrac{37}{2}t^6 - \tfrac{11}{6}t^5 - \tfrac{14}{3}t^4 + 8t^3 + \tfrac{7}{6}t^2 - \tfrac{11}{6}t + 1, \\
i(\mathcal{P}_{+++},t) &= \tfrac{143}{720}t^7 + \tfrac{1}{180}t^6 - \tfrac{977}{360}t^5 + \tfrac{7}{18}t^4 + \tfrac{7967}{720}t^3 + \tfrac{469}{180}t^2 - \tfrac{91}{20}t + 1, \\
i( \mathcal{P}_{+-+},t) &= 29401442376075000t^7 + \tfrac{241325060195043}{20}t^6 - \tfrac{440146708947}{40}t^5 \\
&\quad + \tfrac{147110412735}{8}t^4 - \tfrac{39629403}{8}t^3 + \tfrac{40734679}{40}t^2 - \tfrac{59}{20}t + 1.
\end{align*}
This completes the proof.
\end{proof}

Finally, one might inquire whether zero can occur as an Ehrhart coefficient.
Indeed, this is possible, as demonstrated by the Reeve tetrahedron $\mathcal{T}_{12}$ in Example \ref{ReeveTetrahedron}, which is Ehrhart positive.
We further provide an example of a polytope that is not Ehrhart positive but has a vanishing middle coefficient.
Let $\mathcal{P} = \mathrm{Pry}^1(\mathcal{T}_{48}) \times \mathcal{T}_{20}$.
Its Ehrhart polynomial is:
\[
i(\mathcal{P},t) = \tfrac{20}{3}t^7 + \tfrac{148}{9}t^6 - \tfrac{187}{18}t^4 + \tfrac{13}{2}t^3 + \tfrac{53}{18}t^2 - \tfrac{19}{6}t + 1,
\]
where the coefficient of $t^5$ is zero.

\section{Embedding Theorems (II--V) and Theirs Applications}

\subsection{Embedding Theorems (II--V)}

Inspired by Embedding Theorem I, we establish the following general Embedding Theorem II.

\begin{thm}[Embedding Theorem II]\label{Embedding-Them-II}
Let $\mathcal{P}_1$ be a $d_1$-dimensional integral polytope with Ehrhart polynomial
$i(\mathcal{P}_1,t) = 1 + \sum_{i=1}^{d_1} c_i t^i$, where each coefficient $c_i$ is non-zero.
Let $\mathcal{P}_2$ be a $d_2$-dimensional integral polytope with Ehrhart polynomial
$i(\mathcal{P}_2,t) = 1 + \sum_{i=1}^{d_2} a_i t^i$, where each coefficient $a_i$ is non-zero.
If $d_1 \geq d_2 \geq 2$ are integers, then there exist two $(d_1+d_2)$-dimensional integral polytopes $\mathcal{Q}_1$ and $\mathcal{Q}_2$
with Ehrhart polynomials given by
\begin{align*}
i(\mathcal{Q}_1, t) = 1 + \sum_{i=1}^{d_1+d_2} h_i t^i \quad \text{and} \quad i(\mathcal{Q}_2, t) = 1 + \sum_{i=1}^{d_1+d_2} g_i t^i,
\end{align*}
such that the following conditions hold:
\begin{align*}
&\mathrm{sgn}(h_i) = \mathrm{sgn}(c_i) \quad \text{for all} \quad 1 \leq i \leq d_1; \\
&\mathrm{sgn}(h_i) = \mathrm{sgn}(a_{i-d_1}) \quad \text{for all} \quad d_1+1 \leq i \leq d_1+d_2-2; \\
&\text{in particular,} \quad \mathrm{sgn}(h_{d_1-1}) = \mathrm{sgn}(h_{d_1}) = +1;
\end{align*}
and
\begin{align*}
&\mathrm{sgn}(g_i) = \mathrm{sgn}(a_i) \quad \text{for all} \quad 1 \leq i \leq d_2; \\
&\mathrm{sgn}(g_i) = \mathrm{sgn}(c_{i-d_2}) \quad \text{for all} \quad d_2+1 \leq i \leq d_1+d_2-2; \\
&\text{in particular,} \quad \mathrm{sgn}(g_{d_2-1}) = \mathrm{sgn}(g_{d_2}) = +1.
\end{align*}
\end{thm}

\begin{proof}
We construct the polytopes $\mathcal{Q}_1$ and $\mathcal{Q}_2$ as follows:
\begin{align*}
\mathcal{Q}_1 = r\mathcal{P}_1 \times \mathcal{P}_2 \quad \text{and} \quad \mathcal{Q}_2 = \mathcal{P}_1 \times r\mathcal{P}_2,
\end{align*}
where $r$ is a sufficiently large positive integer. By Proposition \ref{prop_Cartesian}, their Ehrhart polynomials satisfy
\begin{align*}
i(\mathcal{Q}_1, t) &= \left( 1 + \sum_{i=1}^{d_1} r^i c_i t^i \right) \cdot \left( 1 + \sum_{i=1}^{d_2} a_i t^i \right) \\
&= 1 + \sum_{i=1}^{d_2} \left( c_i r^i + c_{i-1} r^{i-1} a_1 + \cdots + a_i \right) t^i + \sum_{i=d_2+1}^{d_1} \left( c_i r^i + c_{i-1} r^{i-1} a_1 + \cdots + c_{i-d_2} r^{i-d_2} a_{d_2} \right) t^i \\
&\quad + \sum_{i=d_1+1}^{d_1+d_2} \left( c_{d_1} r^{d_1} a_{i-d_1} + c_{d_1-1} r^{d_1-1} a_{i-d_1+1} + \cdots + c_{i-d_2} r^{i-d_2} a_{d_2} \right) t^i.
\end{align*}
Since $c_{d_1} > 0$, $c_{d_1-1} > 0$, $a_{d_2} > 0$, and $a_{d_2-1} > 0$, for sufficiently large $r$, the sign conditions for $h_i$ are satisfied.

Similarly, for $\mathcal{Q}_2$, we have
\begin{align*}
i(\mathcal{Q}_2, t) &= \left( 1 + \sum_{i=1}^{d_1} c_i t^i \right) \cdot \left( 1 + \sum_{i=1}^{d_2} r^i a_i t^i \right) \\
&= 1 + \sum_{i=1}^{d_2} \left( c_i + c_{i-1} a_1 r + \cdots + a_i r^i \right) t^i + \sum_{i=d_2+1}^{d_1} \left( c_i + c_{i-1} a_1 r + \cdots + c_{i-d_2} a_{d_2} r^{d_2} \right) t^i \\
&\quad + \sum_{i=d_1+1}^{d_1+d_2} \left( c_{d_1} a_{i-d_1} r^{i-d_1} + c_{d_1-1} a_{i-d_1+1} r^{i-d_1+1} + \cdots + c_{i-d_2} a_{d_2} r^{d_2} \right) t^i.
\end{align*}
By the same reasoning (using $c_{d_1} > 0$, $c_{d_1-1} > 0$, $a_{d_2} > 0$, and $a_{d_2-1} > 0$), the sign conditions for $g_i$ hold for sufficiently large $r$. This completes the proof.
\end{proof}


\begin{cor}
Let $\mathcal{P}$ be a $d$-dimensional integral polytope with Ehrhart polynomial
$i(\mathcal{P},t)= 1+\sum_{i=1}^{d} c_i t^i$, where each $c_i \neq 0$.
Then there exist $(d+3)$-dimensional integral polytopes $\mathcal{Q}_1$ and $\mathcal{Q}_2$ with Ehrhart polynomials
\begin{align*}
i(\mathcal{Q}_1, t) = 1 + \sum_{i=1}^{d+3} h_i t^i \quad \text{and} \quad i(\mathcal{Q}_2, t) = 1 + \sum_{i=1}^{d+3} g_i t^i
\end{align*}
satisfying the sign conditions:
\begin{align*}
\mathrm{sgn}(h_i) &=\mathrm{sgn}(c_i)\quad\quad  \text{for all }\quad 1\leq i\leq d, \\
\mathrm{sgn}(h_{d+1})&=-1, \\
\mathrm{sgn}(h_{d-1})&=\mathrm{sgn}(h_d)=+1;
\end{align*}
and
\begin{align*}
\mathrm{sgn}(g_i) &=\mathrm{sgn}(c_{i-3})\quad\quad  \text{for all }\quad 4\leq i\leq d+1, \\
\mathrm{sgn}(g_1)&=-1, \\
\mathrm{sgn}(g_2)&=\mathrm{sgn}(g_3)=+1.
\end{align*}
\end{cor}

\begin{proof}
Take $\mathcal{P}_2$ as the Reeve tetrahedron $\mathcal{T}_m$ with $m>12$ in Theorem \ref{Embedding-Them-II}. The conclusion follows immediately.
\end{proof}

\begin{thm}[Embedding Theorem III]\label{Embedding-III}
Let $\mathcal{P}$ be a $d$-dimensional integral polytope with Ehrhart polynomial
$i(\mathcal{P},t)= 1+\sum_{i=1}^{d} c_i t^i$, where each $c_i \neq 0$.
Then there exists a $(d+2)$-dimensional integral polytope $\mathcal{Q}$ with Ehrhart polynomial
\begin{align*}
i(\mathcal{Q}, t) = 1 + \sum_{i=1}^{d+2} h_i t^i
\end{align*}
satisfying the sign conditions:
\begin{align*}
\mathrm{sgn}(h_i) &= \mathrm{sgn}(c_{i-2}) \quad\quad \text{for all }\quad 4\leq i\leq d, \\
\mathrm{sgn}(h_2) &= +1, \\
\mathrm{sgn}(h_1) &= \mathrm{sgn}(h_3) = \mathrm{sgn}(c_1).
\end{align*}
\end{thm}

\begin{proof}
Consider an integral convex polygon $\mathcal{U}$ with area $a$ and $b$ boundary lattice points. By \cite[Theorem 2.9]{BeckRobins}, its Ehrhart polynomial is
\begin{align*}
i(\mathcal{U},t) = a t^2 + \frac{b}{2} t + 1.
\end{align*}
We may choose sufficiently large $a$ while keeping $b$ relatively small; for example, take $\mathcal{U} = \conv\{(0,0), (1,0), (1,a), (2,a)\}$, which yields $b=4$. See Pick's Theorem \cite[Theorem 2.8]{BeckRobins}.

Define the polytope
\begin{align*}
\mathcal{Q} = r\mathcal{P} \times \mathcal{U}
\end{align*}
where $r$ is a sufficiently large integer. Applying Proposition \ref{prop_Cartesian} gives
\begin{align*}
i(\mathcal{Q}, t) &= 1 + \left(c_1r + \frac{b}{2}\right)t + \left(c_2r^2 + \frac{c_1 r b}{2} + a\right)t^2 + c_d r^d a  t^{d+2} \\
&\quad + \left( \frac{c_d r^d b}{2} + c_{d-1} r^{d-1} a \right) t^{d+1} \\
&\quad + \sum_{i=3}^d \left( c_i r^i + \frac{c_{i-1} r^{i-1} b}{2} + c_{i-2} r^{i-2} a \right) t^i.
\end{align*}
For integers $a \gg r \gg 0$, the sign patterns satisfy:
\begin{align*}
\mathrm{sgn}\left( c_i r^i + \frac{c_{i-1} r^{i-1} b}{2} + c_{i-2} r^{i-2} a \right) &= \mathrm{sgn}(c_{i-2}) \quad\quad (3 \leq i \leq d), \\
\mathrm{sgn}\left( c_1r + \frac{b}{2} \right) &= \mathrm{sgn}(c_1), \\
\mathrm{sgn}\left( c_2r^2 + \frac{c_1 r b}{2} + a \right) &= \mathrm{sgn}(a) = +1.
\end{align*}
This establishes the required sign conditions.
\end{proof}

\begin{thm}[Embedding Theorem IV]\label{Embedding-IV}
Let $\mathcal{P}$ be a $d$-dimensional integral polytope with Ehrhart polynomial $i(\mathcal{P},t) = 1 + \sum_{i=1}^{d} c_i t^i$, where each $c_i \neq 0$. Then there exists a $(d+3)$-dimensional integral polytope $\mathcal{Q}$ with Ehrhart polynomial
\[
i(\mathcal{Q}, t) = 1 + \sum_{i=1}^{d+3} h_i t^i,
\]
such that the following sign conditions hold:
\[
\mathrm{sgn}(h_i) = -\mathrm{sgn}(c_{i-1}) \quad \text{for all} \quad 2 \leq i \leq d; \quad \mathrm{sgn}(h_d) = \mathrm{sgn}(h_{d+1}) = -1; \quad \text{and} \quad \mathrm{sgn}(h_1) = -1.
\]
\end{thm}

\begin{proof}
We construct the polytope $\mathcal{Q}$ as $\mathcal{Q} = r\mathcal{P} \times \mathcal{T}_m$, where $r$ and $m$ are sufficiently large integers to be chosen later. By Proposition~\ref{prop_Cartesian}, the Ehrhart polynomial of $\mathcal{Q}$ is given by

\begin{footnotesize}
\begin{align*}
i(\mathcal{Q}, t) &=1+\left(c_1r+\frac{12-m}{6}\right)t+\left(c_2r^2+\frac{c_1r(12-m)}{6}+1\right)t^2+ \left(c_3r^3+\frac{c_2r^2(12-m)}{6}+c_1r+\frac{m}{6}\right) t^{3}
\\&\quad +\sum_{i=4}^d\left(c_ir^i+\frac{c_{i-1}r^{i-1}(12-m)}{6}+c_{i-2}r^{i-2}+\frac{c_{i-3}r^{i-3}m}{6}\right)t^i
\\&\quad +\left( \frac{c_dr^d(12-m)}{6}+c_{d-1}r^{d-1}+\frac{c_{d-2}r^{d-2}m}{6}\right) t^{d+1} +\left( c_dr^d+\frac{c_{d-1}r^{d-1}m}{6}\right)t^{d+2} +\frac{c_dr^dm}{6}t^{d+3}.
\end{align*}
\end{footnotesize}

Now, select $r$ and $m$ such that $m \gg r \gg 0$ (i.e., $r$ is a sufficiently large positive integer and $m$ is sufficiently large relative to $r$). Under this choice, the dominant terms govern the signs of the coefficients as follows:
\begin{itemize}
    \item For $h_1$: $\mathrm{sgn}\left(c_1 r + \frac{12 - m}{6}\right) = -1$, since the term $-\frac{m}{6}$ dominates for large $m$.
    \item For $h_2$: $\mathrm{sgn}\left(c_2 r^2 + \frac{c_1 r (12 - m)}{6} + 1\right) = \mathrm{sgn}\left( \frac{c_1 r (12 - m)}{6} \right) = -\mathrm{sgn}(c_1)$.
    \item For $h_3$: $\mathrm{sgn}\left(c_3 r^3 + \frac{c_2 r^2 (12 - m)}{6} + c_1 r + \frac{m}{6}\right) = \mathrm{sgn}\left( \frac{c_2 r^2 (12 - m)}{6} \right) = -\mathrm{sgn}(c_2)$.
    \item For $4 \leq i \leq d$: $\mathrm{sgn}\left(c_i r^i + \frac{c_{i-1} r^{i-1} (12 - m)}{6} + c_{i-2} r^{i-2} + \frac{c_{i-3} r^{i-3} m}{6}\right) = \mathrm{sgn}\left(\frac{c_{i-1} r^{i-1} (12 - m)}{6}\right) = -\mathrm{sgn}(c_{i-1})$.
    \item For $h_{d+1}$: $\mathrm{sgn}\left( \frac{c_d r^d (12 - m)}{6} + c_{d-1} r^{d-1} + \frac{c_{d-2} r^{d-2} m}{6} \right) = \mathrm{sgn}\left(\frac{c_d r^d (12 - m)}{6}\right) = -1$.
\end{itemize}
This establishes the required sign conditions: $\mathrm{sgn}(h_i) = -\mathrm{sgn}(c_{i-1})$ for $2 \leq i \leq d$, $\mathrm{sgn}(h_d) = \mathrm{sgn}(h_{d+1}) = -1$, and $\mathrm{sgn}(h_1) = -1$. The coefficients $h_{d+2}$ and $h_{d+3}$ are not constrained by the theorem. Therefore, the polytope $\mathcal{Q}$ satisfies the claim.
\end{proof}

\begin{thm}[Embedding Theorem V]\label{Embedding-V}
Let $\mathcal{P}$ be a $d$-dimensional integral polytope with Ehrhart polynomial
$i(\mathcal{P},t)= 1 + \sum_{i=1}^{d} c_i t^i$, where each coefficient $c_i$ is nonzero.
Then there exists a $(d+3)$-dimensional integral polytope $\mathcal{Q}$ with Ehrhart polynomial
\[
i(\mathcal{Q}, t) = 1 + \sum_{i=1}^{d+3} h_i t^i
\]
satisfying the following sign conditions:
\begin{align*}
&\sgn(h_i) = \sgn(c_{i-3}) \quad\quad \text{for all }\quad 4 \leq i \leq d+1, \\
&\sgn(h_3) = +1, \\
&\sgn(h_2) = -\sgn(c_1), \\
&\sgn(h_1) = -1.
\end{align*}
\end{thm}

\begin{proof}
Consider the polytope
\[
\mathcal{Q} = \mathcal{P} \times r\mathcal{T}_m
\]
where $r$ and $m$ are sufficiently large positive integers. By Proposition \ref{prop_Cartesian}, the Ehrhart polynomial is given by:

\begin{footnotesize}
\begin{align*}
i(\mathcal{Q}, t) &=1+\left(c_1+\frac{(12-m)r}{6}\right)t+\left(c_2+\frac{c_1r(12-m)}{6}+r^2\right)t^2+ \left(c_3+\frac{c_2r(12-m)}{6}+c_1r^2+\frac{mr^3}{6}\right) t^{3}
\\&\quad +\sum_{i=4}^d\left(c_i+\frac{c_{i-1}r(12-m)}{6}+c_{i-2}r^2+\frac{c_{i-3}r^3m}{6}\right)t^i
\\&\quad +\left( \frac{c_dr(12-m)}{6}+c_{d-1}r^{2}+\frac{c_{d-2}r^3m}{6}\right) t^{d+1} +\left( c_dr^2+\frac{c_{d-1}r^{3}m}{6}\right)t^{d+2} +\frac{c_dr^3m}{6}t^{d+3}.
\end{align*}
\end{footnotesize}
Taking $r$ sufficiently large and $m$ sufficiently large relative to $r$ (i.e., $m \gg r \gg 0$), we obtain the dominant terms:
\begin{align*}
&\sgn\left(c_i + \frac{c_{i-1} r (12-m)}{6} + c_{i-2} r^2 + \frac{c_{i-3} r^3 m}{6}\right)
= \sgn\left(\frac{c_{i-3} r^3 m}{6}\right)
&& \text{for } 4 \leq i \leq d, \\
&\sgn\left( \frac{c_d r (12-m)}{6} + c_{d-1} r^{2} + \frac{c_{d-2} r^3 m}{6} \right)
= \sgn\left(\frac{c_{d-2} r^3 m}{6}\right), \\
&\sgn\left(c_3 + \frac{c_2 r (12-m)}{6} + c_1 r^2 + \frac{m r^3}{6}\right)
= \sgn\left(\frac{m r^3}{6}\right) = +1, \\
&\sgn\left(c_2 + \frac{c_1 r (12-m)}{6} + r^2\right)
= \sgn\left( \frac{c_1 r (12-m)}{6} \right) = -\sgn(c_1), \\
&\sgn\left(c_1 + \frac{(12-m)r}{6}\right) = -1.
\end{align*}
This establishes the required sign conditions.
\end{proof}

\subsection{Applications}
Let $d \geq 3$ be a positive integer. For Question~\ref{QuestionHHTY}, Hibi, Higashitani, Tsuchiya, and Yoshida established $d-1$ cases among the $2^{d-2}$ possible sign patterns; see Theorem~\ref{Hibi-HTY-Results}. Consequently, the remaining $2^{d-2} - (d-1)$ cases require verification. The following theorem leverages our Embedding Theorem to demonstrate that only a limited number of cases need be examined.

The \emph{Fibonacci sequence} $\{F_n\}_{n \geq 1}$ is defined by $F_1 = 1$, $F_2 = 1$, and $F_n = F_{n-1} + F_{n-2}$ for $n \geq 3$. We refer to this relation as the \emph{Fibonacci recurrence}.

\begin{thm}\label{Fibonacci-Bound}
Let $d \geq 5$. Suppose Question~\ref{QuestionHHTY} has been verified for polytopes of dimensions $1$ through $d$. Then for $(d+1)$-dimensional polytopes, under the constraints of Embedding Theorems~I--IV, it suffices to verify $F_{d+1}$ cases, where $F_n$ satisfies the Fibonacci recurrence $F_n = F_{n-1} + F_{n-2}$ with initial conditions $F_6 = 1$ and $F_7 = 1$.
\end{thm}
\begin{proof}
Let $\mathcal{P}$ be a $(d+1)$-dimensional integral polytope with Ehrhart polynomial $i(\mathcal{P},t) = 1 + \sum_{i=1}^{d+1} c_i t^i$. Embedding Theorem~I~\ref{thm_Embedding Theorem} restricts consideration to sign patterns of the form
\begin{align*}
\Big\{ (\mathrm{sgn}(c_{d-1}), \ldots, \mathrm{sgn}(c_1))\in\{-1,+1\}^{d-1} \mid \mathrm{sgn}(c_{d-1}) = \mathrm{sgn}(c_1) = -1 \Big\}.
\end{align*}
Embedding Theorem~IV~\ref{Embedding-IV} further restricts to patterns satisfying
\begin{align*}
\Big\{ (\mathrm{sgn}(c_{d-1}), \ldots, \mathrm{sgn}(c_1))\in\{-1,+1\}^{d-1} \mid  \mathrm{sgn}(c_{d-1}) = \mathrm{sgn}(c_1) = -1,\  \mathrm{sgn}(c_{d-2}) = +1 \Big\}.
\end{align*}
Embedding Theorem~III~\ref{Embedding-III} addresses the subcase where $\mathrm{sgn}(c_1, c_2, c_3) = (-1, +1, -1)$, while Embedding Theorem~II~\ref{Embedding-Them-II} eliminates patterns containing consecutive $+1$ entries.

The number of unresolved cases are thus:
\begin{align*}
F_{d+1} = \# \Big\{ (\mathrm{sgn}(c_{d-1}),&\ldots,\mathrm{sgn}(c_1))\in\{-1,+1\}^{d-1} \mid \mathrm{sgn}(c_{d-1}) = \mathrm{sgn}(c_1) = \mathrm{sgn}(c_2) = -1, \\
& \mathrm{sgn}(c_{d-2}) = +1,\ \text{with no consecutive $+1$ entries} \Big\}.
\end{align*}
When $\mathrm{sgn}(c_3) = -1$, we require $F_d$ verifications. If $\mathrm{sgn}(c_3) = +1$, then $\mathrm{sgn}(c_4) = -1$ (to avoid consecutive $+1$), requiring $F_{d-1}$ verifications. This yields the recurrence $F_{d+1} = F_d + F_{d-1}$. Initial conditions $F_6 = 1$ and $F_7 = 1$ complete the proof.
\end{proof}

\begin{prop}
Given integers $i_1,\ldots,i_q$ with $1\leq i_1<\cdots < i_q\leq 6$, there exists an $8$-dimensional integral polytope $\mathcal{P}$ such that the coefficients of $t^{i_1},\ldots,t^{i_q}$ in $i(\mathcal{P}, t)$ are negative, and all remaining coefficients are positive.
\end{prop}
\begin{proof}
Let $\mathcal{P}$ be an $8$-dimensional integral polytope with Ehrhart polynomial
$i(\mathcal{P},t) = 1 + \sum_{i=1}^{8} c_i t^i$, where each $c_i \neq 0$.
By Theorem \ref{Fibonacci-Bound}, it suffices to verify the following sign patterns:
\begin{align*}
&\big(\mathrm{sgn}(c_6),\mathrm{sgn}(c_5),\mathrm{sgn}(c_4),\mathrm{sgn}(c_3),\mathrm{sgn}(c_2),\mathrm{sgn}(c_1)\big) = (-1,+1,-1,+1,-1,-1),\\
&\big(\mathrm{sgn}(c_6),\mathrm{sgn}(c_5),\mathrm{sgn}(c_4),\mathrm{sgn}(c_3),\mathrm{sgn}(c_2),\mathrm{sgn}(c_1)\big) = (-1,+1,-1,-1,-1,-1).
\end{align*}

\noindent
\textit{Case 1:} For the first sign pattern, define
\[\mathcal{P} = \mathcal{Q}_1 \times \mathcal{Q}_2,\]
where
\begin{align*}
\mathcal{Q}_1 &= \mathrm{conv}\left\{(0,0,0),(1,0,0),(0,1,0),(2,2,1000)\right\}, \\
\mathcal{Q}_2 &= \mathrm{conv}\left\{\mathbf{0}^5, \mathbf{e}_1^5, \mathbf{e}_2^5, (0,0,1,0,1), (0,0,0,1,1), (3,4,5,8,754)\right\}.
\end{align*}
The Ehrhart polynomials are
\begin{align*}
i(\mathcal{Q}_1,t) &= \frac{500}{3}t^3 + \frac{3}{2}t^2 + \frac{11}{6}t + 1, \\
i(\mathcal{Q}_2,t) &= \frac{247}{40}t^5 + \frac{13}{24}t^4 - \frac{19}{24}t^3 + \frac{35}{24}t^2 - \frac{143}{60}t + 1.
\end{align*}
Thus, the Ehrhart polynomial of $\mathcal{P}$ is
\[
i(\mathcal{P},t) = \frac{6175}{6} t^{8} + \frac{71669}{720} t^{7} - \frac{10783}{90} t^{6} + \frac{89653}{360} t^{5} - \frac{7127}{18} t^{4} + \frac{118781}{720} t^{3} - \frac{127}{90} t^{2} - \frac{11}{20} t + 1,
\]
confirming the first sign pattern.

\noindent
\textit{Case 2:} For the second sign pattern, define
\[
\mathcal{P} = 2 \left( \mathcal{T}_{10} \times \mathcal{T}_{100} \right) \times \mathrm{conv}\left\{(0,0),(1,0),(1,20),(2,20)\right\},
\]
where $\mathcal{T}_{m}$ is the polytope from Example \ref{ReeveTetrahedron}. The Ehrhart polynomial is
\[
i(\mathcal{P},t) = \frac{320000}{9}t^8 + \frac{137600}{9}t^7 - \frac{8320}{3}t^6 + \frac{5888}{9}t^5 - \frac{1360}{3}t^4 - \frac{5080}{9}t^3 - \frac{440}{9}t^2 - \frac{80}{3}t + 1,
\]
confirming the second sign pattern.
\end{proof}

\begin{prop}
Given integers $i_1,\ldots,i_q$ satisfying $1\leq i_1<\cdots < i_q\leq 7$, there exists a $9$-dimensional integral polytope $\mathcal{P}$ such that in the Ehrhart polynomial $i(\mathcal{P}, t)$, the coefficients of $t^{i_1},\ldots,t^{i_q}$ are negative while all remaining coefficients are positive.
\end{prop}
\begin{proof}
Let $\mathcal{P}$ be a $9$-dimensional integral polytope with Ehrhart polynomial
$i(\mathcal{P},t)= 1+\sum_{i=1}^{9} c_i t^i$, where each $c_i \neq 0$.
By Theorem \ref{Fibonacci-Bound}, it suffices to verify that the following three sign patterns hold:
\begin{align*}
&(\mathrm{sgn}(c_7),\mathrm{sgn}(c_6),\ldots,\mathrm{sgn}(c_2),\mathrm{sgn}(c_1))=(-1,+1,-1,-1,-1,-1,-1),
\\& (\mathrm{sgn}(c_7),\mathrm{sgn}(c_6),\ldots,\mathrm{sgn}(c_2),\mathrm{sgn}(c_1))=(-1,+1,-1,+1,-1,-1,-1),
\\& (\mathrm{sgn}(c_7),\mathrm{sgn}(c_6),\ldots,\mathrm{sgn}(c_2),\mathrm{sgn}(c_1))=(-1,+1,-1,-1,+1,-1,-1).
\end{align*}

\noindent\textit{Case 1.} Consider the polytope
\[
\mathcal{P} = \mathcal{T}_{10} \times \mathcal{T}_{100} \times \ell_{10} \times \operatorname{conv}\{(0,0),(1,0),(1,10),(2,10)\}.
\]
Its Ehrhart polynomial is given by
\begin{small}
\[
i(\mathcal{P},t) = \frac{25000}{9}t^9 + \frac{8000}{3}t^8 - \frac{8150}{9}t^7 + \frac{1000}{9}t^6 - \frac{3287}{9}t^5 - \frac{13409}{9}t^4 - \frac{1082}{3}t^3 - \frac{1304}{9}t^2 - \frac{7}{3}t + 1,
\]
\end{small}
which exhibits the first sign pattern.

\noindent\textit{Case 2.} Define $\mathcal{Q}_1$ and $\mathcal{Q}_2$ as
\begin{align*}
\mathcal{Q}_1 & = \operatorname{conv}\left\{\mathbf{0}, \mathbf{e}_1, \mathbf{e}_2, (0,0,1,0,1), (0,0,0,1,1), (3,4,5,8,754)\right\}, \\
\mathcal{Q}_2 & = \operatorname{conv}\left\{(0,0,0,0),(1,0,0,0),(0,1,0,0),(0,0,1,0),(2,2,310,610)\right\},
\end{align*}
where $\mathbf{0}, \mathbf{e}_1, \mathbf{e}_2$ denote vectors in $\mathbb{R}^5$. The polytope $\mathcal{P} = \mathcal{Q}_1 \times \mathcal{Q}_2$ has Ehrhart polynomial
\begin{small}
\[
i(\mathcal{P},t) = \frac{15067}{96} t^{9} + \frac{31681}{1440} t^{8} - \frac{259}{720} t^{7} + \frac{36763}{720} t^{6} - \frac{77363}{1440} t^{5} + \frac{36809}{1440} t^{4} - \frac{1313}{360} t^{3} - \frac{28}{45} t^{2} - \frac{13}{60} t + 1,
\]
\end{small}
realizing the second sign pattern.

\noindent\textit{Case 3.} For the final pattern, take
\[
\mathcal{P} = \mathcal{T}_{10} \times \mathcal{T}_{100} \times \operatorname{conv}\{(0,0,0),(1,0,0),(0,1,0),(2,2,1000)\}.
\]
The Ehrhart polynomial is
\begin{small}
\[
i(\mathcal{P},t) = \frac{125000}{27} t^{9} + \frac{27875}{9} t^{8} - \frac{52255}{18} t^{7} + \frac{6311}{9} t^{6} - \frac{8819}{18} t^{5} - \frac{21634}{9} t^{4} + \frac{7769}{54} t^{3} - \frac{83}{3} t^{2} - \frac{25}{2} t + 1,
\]
\end{small}
which yields the third sign pattern.

The existence of the required polytope for each case is thereby established.
\end{proof}

\section{Concluding Remark}

The five embedding theorems provide a partial resolution to Question \ref{QuestionHHTY}.
We note that Embedding Theorem I (Theorem \ref{thm_Embedding Theorem}) may alternatively be formulated as a corollary of Embedding Theorem \ref{Embedding-Them-II}.
Through application of these embedding theorems, in low dimensions one need only verify a limited number of sign patterns.
Consider a $d$-dimensional integral polytope $\mathcal{P}$ with Ehrhart polynomial
$i(\mathcal{P},t) = 1 + \sum_{i=1}^{d} c_i t^i$, where each coefficient $c_i \neq 0$.

For $d=10$, Theorem \ref{Fibonacci-Bound} implies that it suffices to verify the sign pattern of the Ehrhart coefficients matches one of the following:
\begin{align*}
&(\mathrm{sgn}(c_{8}),\mathrm{sgn}(c_7),\ldots,\mathrm{sgn}(c_2),\mathrm{sgn}(c_1)) = (-1,+1,-1,-1,-1,-1,-1,-1), \\
&(\mathrm{sgn}(c_{8}),\mathrm{sgn}(c_7),\ldots,\mathrm{sgn}(c_2),\mathrm{sgn}(c_1)) = (-1,+1,-1,-1,-1,+1,-1,-1), \\
&(\mathrm{sgn}(c_{8}),\mathrm{sgn}(c_7),\ldots,\mathrm{sgn}(c_2),\mathrm{sgn}(c_1)) = (-1,+1,-1,-1,+1,-1,-1,-1), \\
&(\mathrm{sgn}(c_{8}),\mathrm{sgn}(c_7),\ldots,\mathrm{sgn}(c_2),\mathrm{sgn}(c_1)) = (-1,+1,-1,+1,-1,-1,-1,-1), \\
&(\mathrm{sgn}(c_{8}),\mathrm{sgn}(c_7),\ldots,\mathrm{sgn}(c_2),\mathrm{sgn}(c_1)) = (-1,+1,-1,+1,-1,+1,-1,-1).
\end{align*}
This reduces the number of sign patterns requiring verification from $2^{10-2}=256$ (accounting for fixed signs of leading coefficients) to merely $5$.

Establishing additional embedding theorems presents an interesting direction for future research.
A related challenging problem involves the explicit construction of integral polytopes realizing the aforementioned sign patterns in low dimensions.
Question \ref{QuestionHHTY} remains open in full generality, and its complete resolution is highly anticipated.






\noindent
{\small \textbf{Acknowledgements:}}
This work is partially supported by the National Natural Science Foundation of China [12571355].

\end{document}